\documentclass[11pt,reqno]{article}

\usepackage[a4paper,margin=1.18in]{geometry}
\usepackage[T1]{fontenc}
\usepackage{lmodern}
\usepackage{microtype}
\usepackage{amsmath,amssymb,amsthm,mathtools}
\usepackage{enumitem}
\usepackage{booktabs}
\usepackage{xcolor}
\usepackage{aliascnt}
\usepackage[colorlinks=true,linkcolor=blue,citecolor=blue,urlcolor=blue]{hyperref}
\usepackage[capitalise,noabbrev]{cleveref}

\allowdisplaybreaks
\numberwithin{equation}{section}

\newtheorem{theorem}{Theorem}[section]

\newaliascnt{proposition}{theorem}
\newtheorem{proposition}[proposition]{Proposition}
\aliascntresetthe{proposition}

\newaliascnt{lemma}{theorem}
\newtheorem{lemma}[lemma]{Lemma}
\aliascntresetthe{lemma}

\newaliascnt{corollary}{theorem}
\newtheorem{corollary}[corollary]{Corollary}
\aliascntresetthe{corollary}

\newaliascnt{remark}{theorem}
\newtheorem{remark}[remark]{Remark}
\aliascntresetthe{remark}

\newaliascnt{example}{theorem}

\aliascntresetthe{example}

\theoremstyle{definition}
\newaliascnt{definition}{theorem}

\aliascntresetthe{definition}

\newaliascnt{Hypothesis}{theorem}
\newtheorem{Hypothesis}[Hypothesis]{Hypothesis}
\aliascntresetthe{Hypothesis}

\crefname{theorem}{Theorem}{Theorems}
\crefname{proposition}{Proposition}{Propositions}
\crefname{lemma}{Lemma}{Lemmas}
\crefname{corollary}{Corollary}{Corollaries}
\crefname{remark}{Remark}{Remarks}
\crefname{example}{Example}{Examples}
\crefname{definition}{Definition}{Definitions}
\crefname{Hypothesis}{Hypothesis}{Hypotheses}

\DeclareMathOperator{\add}{add}
\DeclareMathOperator{\End}{End}
\DeclareMathOperator{\Aut}{Aut}
\DeclareMathOperator{\Hom}{Hom}
\DeclareMathOperator{\Ext}{Ext}
\DeclareMathOperator{\Coker}{Coker}
\DeclareMathOperator{\Ker}{Ker}

\DeclareMathOperator{\Gr}{Gr}
\DeclareMathOperator{\ind}{ind}

\DeclareMathOperator{\proj}{proj}
\DeclareMathOperator{\inj}{inj}
\DeclareMathOperator{\modu}{mod}

\newcommand{\C}{\mathcal C}
\newcommand{\A}{\mathcal A}

\newcommand{\Pcal}{\mathcal P}
\newcommand{\E}{\mathbb E}
\newcommand{\F}{\mathsf F}
\newcommand{\X}{\mathsf X}
\newcommand{\kk}{\mathbb C}
\newcommand{\Ksp}{K_0^{\mathrm{sp}}}

\newcommand{\xto}[1]{\xrightarrow{#1}}
\newcommand{\etri}{\dashrightarrow}
\newcommand{\bx}{\mathbf x}

\begin{document}

\title{\bf A Refined Multiplication Formula in
$2$-Calabi--Yau Frobenius Extriangulated Categories}
\author{Ming Ding, Fan Xu and Panyue Zhou}

\maketitle

\begin{abstract}
We prove a multiplication formula for  the Wang--Wei--Zhang cluster character on a Hom-finite $2$-Calabi--Yau Frobenius extriangulated category $\C$ with a cluster-tilting object under the  assumption that its stable category has constructible cones with respect to the induced cluster-tilting object.  This formula generalizes those obtained by Wang--Wei--Zhang for $\C$ in the one-dimensional case and by Keller--Plamondon--Qin for (stably) $2$--Calabi--Yau Frobenius or triangulated categories.  As a consequence, we express the frozen term in an Auslander--Reiten mesh with the character of a minimal cosyzygy
middle term.  For acyclic quivers with principal coefficients, we compute these frozen multiplicities in Higgs categories and obtain explicit principal coefficient mesh relations.
\end{abstract}

\vspace{2mm}


\section{Introduction}
Multiplication formulas for cluster characters express products of cluster characters in terms of geometric invariants of extension spaces.  The starting point is the Caldero--Chapoton character introduced by Caldero and Chapoton \cite{CalderoChapoton}, which associates Laurent polynomials to objects of cluster categories by means of Euler characteristics of quiver Grassmannians.  For cluster categories, Caldero and Keller \cite{CK,CalderoKeller} established two fundamental multiplication formulas: the first is a formula of Hall type for arbitrary extension dimension in Dynkin quivers, while the second is  the exchange relation for a one-dimensional extension space in acyclic quivers.  For arbitrary acyclic quivers, Xiao and Xu \cite{XiaoXu} established a multiplication formula using a geometric projective version of Green's formula, and Xu \cite{XuAcyclic} gave a proof using Calabi--Yau duality and higher associativity. In a parallel direction, Geiss, Leclerc and Schr\"oer \cite{GLS} obtained analogous multiplication formulas for module categories over preprojective algebras.

Palu \cite{Palu2008} then extended cluster characters to Hom-finite $2$-Calabi--Yau triangulated categories with cluster tilting objects.  In the exact setting, Fu and Keller \cite{FuKeller} constructed cluster characters on Hom-finite $2$-Calabi--Yau Frobenius exact categories.  Their construction incorporates coefficients through the projective--injective summands and satisfies the exchange multiplication relation when the relevant extension space is one-dimensional.  Palu \cite{Palu2012} subsequently established, in the triangulated setting, a multiplication formula for arbitrary pairs of objects, with coefficients given by Euler characteristics of constructible subsets of projectivised extension spaces.  His theorem simultaneously generalizes the finite type Caldero--Keller formula and the acyclic formulas of Xiao--Xu and Xu.  More precisely, for objects $L,M$ in a $2$-Calabi--Yau triangulated category $\A$ with a cluster tilting object, his formula is
\[
 \chi\bigl(\mathbb P\Hom_{\A}(L,\Sigma M)\bigr)X_LX_M
 =\int_{\mathbb P\Hom_{\A}(L,\Sigma M)}X_{\operatorname{mt}(\delta)}\,d\chi
 +\int_{\mathbb P\Hom_{\A}(M,\Sigma L)}X_{\operatorname{mt}(\eta)}\,d\chi.
\]

In this setting, Dominguez and Geiss \cite{DominguezGeiss} proved a particularly useful mesh relation: for an Auslander--Reiten triangle
\[
  \Sigma Z\longrightarrow Y\longrightarrow Z\longrightarrow\Sigma^2 Z,
\]
one has
\[
  X^T_{\Sigma Z}X^T_Z=X^T_Y+1.
\]
This formula goes beyond the usual one-dimensional exchange formula, since the ambient extension space may have dimension greater than one.  In that situation the multiplication formula for the full extension space does not isolate the line generated by the Auslander--Reiten extension.

Recently, Keller, Plamondon and Qin \cite{KellerPlamondonQin} refined Palu's formula by replacing the full extension space with an arbitrary nonzero linear subspace $V$.  The reverse term is then taken over those extension classes which do not annihilate $V$ under the $2$-Calabi--Yau pairing.  Keller, Plamondon and Qin \cite[Section~3]{KellerPlamondonQin} also treated Frobenius exact categories using the Fu--Keller character.  This refinement for arbitrary subspaces by Keller, Plamondon and Qin \cite{KellerPlamondonQin} provides precisely the additional flexibility needed to isolate the Auslander--Reiten extension line of Dominguez and Geiss \cite{DominguezGeiss} when the extension space has dimension greater than one.

Extriangulated categories, introduced by Nakaoka and Palu \cite{NakaokaPalu}, provide a common framework for exact and triangulated categories.  If an extriangulated category is Frobenius, its stable category is triangulated.  Wang, Wei and Zhang \cite{WangWeiZhangCC} extended the Palu and Fu--Keller constructions to $2$-Calabi--Yau Frobenius extriangulated categories. Their cluster character retains projective--injective objects as frozen variables and satisfies the one-dimensional exchange multiplication formula.  Thus the triangulated theory of Palu and the Frobenius exact theory of Fu--Keller fit into a common extriangulated framework.  Keller, Plamondon and Qin \cite{KellerPlamondonQin} explicitly observed that their refined Frobenius formula should extend further to suitable extriangulated categories, including Higgs categories.  We establish an extriangulated version for the Wang--Wei--Zhang character. In the Hom-finite exact case, Faber, Marsh and Pressland \cite[Proposition~5.4]{FaberMarshPressland} identify this character with the Fu--Keller character before frozen specialisation. Thus our formula recovers the Hom-finite Frobenius case of Keller--Plamondon--Qin, see \cref{rem:exact-comparison}.

Two other comparisons are relevant. Fraser, Keller and Wu \cite[Section~5]{FraserKellerWu} constructed a character on Higgs categories and proved its one-dimensional multiplication relation. Grabowski and Pressland \cite[Propositions~5.22 and~5.28]{GrabowskiPressland} developed cluster characters and their behaviour under partial stabilisation in an extriangulated framework. Our result concerns arbitrary subspaces of extension spaces for the explicit character used here. The stable incidence geometry is inherited from Palu and Keller--Plamondon--Qin. The additional step is to establish constructibility and exponent identities in the full split Grothendieck group, including the frozen coordinates.

Let $\C$ be a Hom-finite $2$-Calabi--Yau Frobenius extriangulated category with a basic cluster tilting object $T$ and let $X^T$ be the cluster character of Wang, Wei and Zhang \cite{WangWeiZhangCC}.  The $2$-Calabi--Yau structure gives a nondegenerate pairing
\[
 \beta_{L,M}\colon\E(L,M)\times\E(M,L)\longrightarrow\kk.
\]
For a nonzero subspace $V\subseteq\E(L,M)$, put
\[
 V^{\perp}=\{\eta\in\E(M,L)\mid \beta_{L,M}(v,\eta)=0
 \text{ for every }v\in V\}
\]
and
\[
 \mathcal R_V=\mathbb P\E(M,L)\setminus\mathbb P(V^{\perp}).
\]
Assume that the stable category $\A=\underline{\C}$ has constructible cones with respect to $T'$ in the sense of Palu.  Under this hypothesis, our main result states that
\begin{equation}\label{eq:intro-main}
\begin{aligned}
 \chi(\mathbb PV)X^T_{L}X^T_M
 ={}&\sum_Y\chi\bigl(\mathbb P(V\cap\E(L,M)_{\langle Y\rangle})\bigr)X^T_Y\\
 &+\sum_Y\chi\bigl(\mathcal R_V\cap
        \mathbb P\E(M,L)_{\langle Y\rangle}\bigr)X^T_Y.
\end{aligned}
\end{equation}
Both sums are finite.  No correction term has been suppressed in \eqref{eq:intro-main}, since a projective--injective summand of a middle term is part of that middle term and hence contributes its frozen monomial to $X^T_Y$.  The incidence geometry itself lives in the stable category. The additional extriangulated input is the index defect formula, which recovers the full index of an actual middle term and therefore retains its projective--injective coordinates.  In particular, if
\[
 \tau Z\longrightarrow Y\longrightarrow Z\dashrightarrow,
 \quad
 Z\longrightarrow I_Z\longrightarrow\tau Z\dashrightarrow
\]
are an Auslander--Reiten $\E$-triangle and its minimal companion cosyzygy, then
\[
 [I_Z]
 =\ind_T Z+\ind_T(\tau Z)-\Phi([H(\tau Z)]),
\]
so the coefficient term in the mesh relation is determined entirely by the index-defect data.

For the Higgs category attached to an acyclic quiver with principal coefficients, we further compute the frozen multiplicities directly from the relative Ginzburg dg algebra.  If $M$ is an indecomposable non-projective $A=\kk Q$-module and $Z_M$ is its Higgs lift, then the frozen term in the Auslander--Reiten mesh at $Z_M$ is
\[
 \prod_{i\in Q_0}x_{i^+}^{\dim(\tau_A M)_i}.
\]
Thus the refined formula gives an explicit principal coefficient mesh relation for every such $M$.  Dupont \cite[Definition~2.1 and Proposition~2.2]{DupontCoefficients} already defined the corresponding cluster character with coefficients and proved the Auslander--Reiten multiplication formula for hereditary modules. We therefore record the comparison only in a remark. The principal coefficient mesh identity is accompanied here by an explicit calculation of the projective--injective object which realises its coefficient term. In homogeneous tubes the ambient Higgs extension space has dimension two. Our proof selects the Auslander--Reiten line inside that space.

The paper is organised as follows.  In \cref{sec:background}, we recall Frobenius extriangulated categories, indices and cluster characters.   In \cref{multiplication}, we discuss constructibility, construct the incidence varieties and establish a multiplication formula.  We conclude in \cref{Higgs} with Jacobi-finite Higgs categories and the principal coefficient calculation, including explicit Auslander--Reiten mesh formulas.

We work over the field $\mathbb C$.  Endomorphism algebras have multiplication given by composition, and all modules are right modules unless stated otherwise. Throughout the paper, all categories are $\kk$-linear, Krull--Schmidt and skeletally small.  All Hom and extension spaces under consideration are finite-dimensional.  A Krull--Schmidt additive category is idempotent complete. Euler characteristic means the topological Euler characteristic with compact support.

\section{Preliminaries}\label{sec:background}

\subsection{Extriangulated and stable categories}

An extriangulated category is a triple $(\C,\E,\mathfrak s)$ consisting of an additive category $\C$, a additive bifunctor
\[
 \E\colon\C^{\mathrm{op}}\times\C\longrightarrow {\rm Ab}
\]
and an additive realisation $\mathfrak s$ satisfying the axioms of Nakaoka and Palu \cite{NakaokaPalu}.  If $\delta\in\E(C,A)$ is realised by $A\xto{x}B\xto{y}C$, we write
\[
 A\xto{x}B\xto{y}C\overset{\delta}{\etri}
\]
and call it an $\E$-triangle.  Exact sequences in an exact category and distinguished triangles in a triangulated category are the standard examples.

An object $P$ is projective if $\E(P,-)=0$, and an object $I$ is injective if $\E(-,I)=0$.  The category is Frobenius if it has enough projectives and enough injectives and the two classes coincide.  We denote their common subcategory by $\Pcal$.  The stable category
\[
 \A=\underline{\C}=\C/[\Pcal]
\]
is triangulated by Nakaoka and Palu \cite{NakaokaPalu}.  We write $\pi\colon\C\to\A$ for the quotient functor and $\Sigma$ for the suspension of $\A$.  There are natural isomorphisms
\[
 \E(X,Y)\cong\Hom_{\A}(\pi X,\Sigma\pi Y).
\]

Assume that $\C$ is $2$-Calabi--Yau in the sense of Chang, Zhou and Zhu \cite{CZZ}, meaning that it is equipped with bifunctorial isomorphisms
\begin{equation}\label{eq:2cy}
 \E(X,Y)\cong D\E(Y,X),
 \end{equation}
for any $X,Y\in \C$, where $D=\Hom_{\kk}(-,\kk)$.
Equivalently, there are nondegenerate bifunctorial pairings
\[
 \beta_{X,Y}\colon\E(X,Y)\times\E(Y,X)\longrightarrow\kk.
\]
The stable category $\A$ is then a Hom-finite $2$-Calabi--Yau triangulated category.

\subsection{Cluster tilting objects and the index}
For an object $T$ of $\C$,  denote by $\add T$  the full subcategory of direct summands of finite direct
sums of copies of $T$.  More generally, for a full additive subcategory
$\mathcal U\subseteq\C$, a morphism $U\to X$ with $U\in\mathcal U$ is a right
$\mathcal U$-approximation if every morphism from an object of $\mathcal U$ to
$X$ factors through it.  Contravariant finiteness means that every object has
a right approximation.  Covariant finiteness is defined dually, and
\emph{functorially finite} means both.

Following Wang, Wei and Zhang \cite{WangWeiZhangCC}, an object $T$ is cluster tilting
if $\add T$ is functorially finite and
\[
 \add T
 =\{X\in\C\mid\E(T,X)=0\}
 =\{X\in\C\mid\E(X,T)=0\}.
\]
This also supplies the stronger approximation triangles often built into the
definition.  Indeed, given a right $\add T$-approximation $f\colon T_0\to X$,
choose a deflation $p\colon P\to X$ with $P$ projective.  Since
$P\in\add T$, the map $(f,p)\colon T_0\oplus P\to X$ is both a right
$\add T$-approximation and a deflation, hence occurs in an $\E$-triangle
\[
 K\longrightarrow T_0\oplus P\longrightarrow X\dashrightarrow.
\]
Applying $\Hom_{\C}(T,-)$ shows that $\E(T,K)=0$, since the approximation is surjective on $\Hom_{\C}(T,-)$ and $\E(T,T_0\oplus P)=0$. Thus $K\in\add T$. The dual argument uses enough injectives.  Thus $\add T$ is strongly
functorially finite in the standard extriangulated sense, in agreement with Wang, Wei and Zhang \cite{WangWeiZhangCC}.
Every projective--injective object belongs to $\add T$.  We fix a basic cluster tilting object
\[
 T=T_1\oplus\cdots\oplus T_r\oplus T_{r+1}\oplus\cdots\oplus T_n,
\]
where $T_i$ is non-projective for $i\leq r$ and projective--injective for $i>r$.  Put
\[
 T'=\pi T=\bigoplus_{i=1}^r\pi T_i,
\hspace{2.5mm}
 \Lambda=\End_{\A}(T'),
\hspace{2.5mm}
 G=\Hom_{\A}(T',-).
\]
The right $\Lambda$-action on $G(X)$ is $h\cdot a=h\circ a$. The object $T'$ is cluster tilting in $\A$ by Wang, Wei and Zhang \cite[Lemma~3.1]{WangWeiZhangCC}.

For an additive category $\mathcal U$, its split Grothendieck group
$K_0^{\mathrm{sp}}(\mathcal U)$ is generated by isomorphism classes $[U]$ modulo
$[U\oplus V]=[U]+[V]$.  If $B$ is a finite-dimensional $\kk$-algebra, then
$\modu B$ denotes the category of finite-dimensional right $B$-modules and
$K_0(\modu B)$ its ordinary Grothendieck group.

Every $M\in\C$ occurs in an $\E$-triangle
\begin{equation}\label{eq:T-presentation}
 T_1^M\longrightarrow T_0^M\longrightarrow M\etri,
~~~\mbox{where}~~T_0^M,T_1^M\in\add T.
\end{equation}
The index
\[
 \ind_T(M)=[T_0^M]-[T_1^M]\in\Ksp(\add T)
\]
is independent of the choice of \eqref{eq:T-presentation} by Wang, Wei and Zhang \cite{WangWeiZhangCC}.  For
$\alpha=\sum_i a_i[T_i]$, set $\bx^\alpha=\prod_i x_i^{a_i}$.

Let $H=G\pi$.  Given $M\in\C$, choose an $\E$-triangle
$M^-\to P\to M\etri$ with $P\in\Pcal$ and set
\[
 \Theta(M)=\ind_T(M^-)-[P]+\ind_T(M).
\]
Wang, Wei and Zhang \cite[Section~3, especially Theorem~3.9]{WangWeiZhangCC} proved that $\Theta(M)$ depends only on the $\Lambda$-module $H(M)$ and  the resulting map descends to the Grothendieck group.  Thus one has a homomorphism
\[
 \Phi\colon K_0(\modu\Lambda)\longrightarrow\Ksp(\add T),
\hspace{2.5mm}
 \Phi([H(M)])=\Theta(M).
\]
The following result is the index defect formula of
Wang, Wei and Zhang.

\begin{proposition}{\rm \cite[Proposition~3.7]{WangWeiZhangCC}}\label{prop:index-defect}
For every $\E$-triangle $A\xto{f}B\xto{g}C\etri$, one has
\[
  \ind_T(A)-\ind_T(B)+\ind_T(C)
  =\Phi\bigl([\Coker H(g)]\bigr).
\]
\end{proposition}

\subsection{The cluster character}

For $M\in\C$, define the finite-dimensional $\Lambda$-module
\[
 \F M=G(\Sigma\pi M).
\]
If $e\in K_0(\modu\Lambda)$, let $\Gr_e(\F M)$ be the projective variety of submodules of $\F M$ with class $e$.  The cluster character defined by Wang, Wei and Zhang \cite{WangWeiZhangCC} is
\begin{equation}\label{eq:character}
 \X^T_M
 =\bx^{\ind_T(M)}
   \sum_e\chi\bigl(\Gr_e(\F M)\bigr)\bx^{-\Phi(e)}.
\end{equation}

Note that
\[
  \X^T_{T_i}=x_i,
 \hspace{2.5mm}
  \X^T_{M\oplus N}=\X^T_M\X^T_N.
\]

If $P=\bigoplus_{i=r+1}^nT_i^{m_i}$ is projective--injective, then
\[
 \X^T_P=\prod_{i=r+1}^n x_i^{m_i}.
\]
The variables $x_{r+1},\ldots,x_n$ are therefore the frozen variables.

When $\dim\E(L,M)=1$, the exchange formula holds by Wang, Wei and Zhang \cite[Theorem~4.4]{WangWeiZhangCC}.

\begin{remark}[Comparison in the exact case]\label{rem:exact-comparison}
Suppose that $\C$ is Frobenius exact, and let $X^{\mathrm{FK}}$ be the
Fu--Keller character. Faber, Marsh and Pressland
\cite[Proposition~5.4]{FaberMarshPressland} prove that the Wang--Wei--Zhang
character coincides with $X^{\mathrm{FK}}$ whenever the endomorphism algebra
of $T$ is Noetherian. This hypothesis is automatic here because $\C$ is
Hom-finite. Thus
\[
 \X^T_M=X^{\mathrm{FK}}_M\hspace{2.5mm}(M\in\C),
\]
with all frozen variables retained. Wang, Wei and Zhang
\cite[Remark~4.10]{WangWeiZhangCC} explicitly record that this result removes
the additional comparison condition appearing earlier in their paper.

Consequently our main formula recovers the Frobenius formula of Keller,
Plamondon and Qin \cite[Theorem~3.6]{KellerPlamondonQin} under the common
Hom-finite hypotheses. If the two formulas use different finite constructible
partitions, passing to a common refinement gives the same Euler integral.
The Hom-infinite framework discussed in that reference is not included in
our standing assumptions.
\end{remark}

\section{A refined multiplication formula}\label{multiplication}
\subsection{Constructibility and extension strata}
A constructible subset of a complex algebraic variety is a finite union of locally closed subsets.  A function $f\colon X\to A$ from a variety to an abelian group is constructible if it has finite image and all its fibres are constructible.  If $S\subseteq X$ is constructible, define
\[
 \int_S f\,d\chi
 =\sum_{a\in f(S)}\chi(S\cap f^{-1}(a))a.
\]
We use the additivity and multiplicativity of Euler characteristic and the fact that every affine space has Euler characteristic one.

For a vector space $U$ and a constructible subset $S\subseteq U$ stable under multiplication by nonzero scalars, we write $\mathbb PS$ for the image of $S\setminus\{0\}$ in $\mathbb PU$.  We use the convention $\mathbb P0=\varnothing$.

We shall also use relative quiver Grassmannians.  If a constructible family
$\rho\colon S\to\operatorname{rep}_{\mathbf d}(\Lambda)$ parametrises
$\Lambda$-modules on fixed vector spaces and $e$ is a dimension vector, then
\[
 \Gr_e(\rho)=\{(s,U)\mid s\in S,\ U\subseteq\rho(s)
                   \text{ is a submodule of dimension vector }e\}.
\]
It is a constructible subset of $S\times\prod_i\Gr_{e_i}(\kk^{d_i})$.
The fibre of its first projection at $s$ is the ordinary quiver Grassmannian
$\Gr_e(\rho(s))$.  Constructibility of its fibres of the Euler characteristic follows
from pushforward of constructible functions. For more details, see Palu \cite[Proposition~2.1 and
Lemma~2.4]{Palu2012}.

Fix $L,M\in\C$.  For each $\delta\in\E(L,M)$, choose a realisation
\begin{equation}\label{eq:realisation}
 M\xto{i_\delta}Y_\delta\xto{p_\delta}L
 \overset{\delta}{\etri}.
\end{equation}
In the stable category this gives a triangle.  Applying $G\Sigma$ yields an exact sequence of $\Lambda$-modules
\begin{equation}\label{eq:four-term}
 \F M\xto{a_\delta}\F Y_\delta\xto{b_\delta}\F L
 \xto{c_\delta}G(\Sigma^2\pi M).
\end{equation}
Changing the realisation of $\delta$ changes this sequence only by an isomorphism of diagrams.

We recall the geometric hypothesis in the form used below.  Let
$\vec A_4$ be the quiver $1\to2\to3\to4$.  For a dimension vector
quadruple $\mathbf d$, write
$\operatorname{rep}^{\Lambda}_{\mathbf d}(\vec A_4)$ for the affine variety of
$\vec A_4$-diagrams in $\modu\Lambda$ on fixed vector spaces of dimensions
$\mathbf d$, and let $\mathrm{GL}(\mathbf d)$ act by change of bases.  Put
\[
 \mathbf d_{\max}=
 \bigl(\dim\F M,\ \dim\F M+\dim\F L,\ \dim\F L,
       \dim G(\Sigma^2\pi M)\bigr),
\]
where inequalities between dimension vectors are componentwise.  The exact
diagram \eqref{eq:four-term} defines a map
\begin{equation}\label{eq:orbit-map}
 \Psi_{L,M}\colon\E(L,M)\longrightarrow
 \coprod_{\mathbf d\leq\mathbf d_{\max}}
 \operatorname{rep}^{\Lambda}_{\mathbf d}(\vec A_4)/
 \mathrm{GL}(\mathbf d),
 \quad
 \delta\longmapsto[\eqref{eq:four-term}].
\end{equation}

Following Palu \cite[Section~1.3]{Palu2012}, the cylinders over extensions from
$L$ to $M$ are \emph{constructible with respect to $T'$} if
\eqref{eq:orbit-map} admits a constructible lift
\[
 \widetilde\Psi_{L,M}\colon\E(L,M)\longrightarrow
 \coprod_{\mathbf d\leq\mathbf d_{\max}}
 \operatorname{rep}^{\Lambda}_{\mathbf d}(\vec A_4).
\]
Thus, after a finite constructible stratification, the four modules and the
three maps in \eqref{eq:four-term} are represented by matrices whose entries
vary regularly with $\delta$.

\begin{Hypothesis}\label{hyp:constructible-cones}
The stable category $\A=\underline{\C}$ has constructible cones with respect to $T'$ in the sense of Palu \cite[Section~1.3]{Palu2012}. Equivalently, for every pair $L,M\in\C$, the cylinders over $\pi L\to\Sigma\pi M$ and $\pi M\to\Sigma\pi L$ are constructible with respect to $T'$ in the preceding sense.
\end{Hypothesis}

\begin{remark}[The one dimensional case]\label{rem:one-dimensional}
Fix $L,M\in\C$ with $\dim_{\kk}\E(L,M)=1$. The part of \cref{hyp:constructible-cones} needed for this pair is automatic. Indeed, $2$-Calabi--Yau duality gives $\dim_{\kk}\E(M,L)=1$. Each extension space is the disjoint union of the zero class and a single nonzero scalar orbit. Nonzero scalar multiples of an extension have isomorphic middle terms, and their four-term module diagrams differ only by rescaling structural maps. Hence the two cylinders and all the incidence loci used for the pair $(L,M)$ are constructible.

Taking $V=\E(L,M)$, both projectivised extension spaces are points and the nondegenerate pairing gives $V^{\perp}=0$. The main formula therefore becomes the one dimensional exchange formula of Wang, Wei and Zhang \cite[Theorem~4.4]{WangWeiZhangCC}. Thus their result is recovered without any additional constructibility assumption in this case.
\end{remark}

\begin{proposition}\label{prop:constructible-lift}
Under \cref{hyp:constructible-cones}, for any fixed pair $L,M\in\C$, the following hold:
\begin{enumerate}[label=\textup{(\roman*)},leftmargin=2.6em]
\item the four-term diagrams \eqref{eq:four-term} and all incidence loci obtained from their submodules, images and inverse images form constructible families.
\item the function $\delta\mapsto\ind_T(Y_\delta)$ has finite image and constructible fibres.
\end{enumerate}
\end{proposition}

\begin{proof}
By \cref{hyp:constructible-cones}, after a finite constructible stratification of $\E(L,M)$ the four modules and three maps in \eqref{eq:four-term} are represented on fixed vector spaces by matrices depending regularly on $\delta$.  Refining the stratification by the ranks of these matrices, kernels and images become algebraic subbundles on each stratum.  Conditions such as $U\subseteq\F Y_\delta$ being a submodule, the class of $b_\delta(U)$, and the class of $a_\delta^{-1}(U)$ are then rank conditions inside products with ordinary Grassmannians.  Hence the corresponding relative quiver Grassmannians and incidence loci are constructible.  This proves (i). Compare Keller, Plamondon and Qin \cite[Lemma~2.12]{KellerPlamondonQin} and Palu \cite[Lemma~3.1]{Palu2012}.

Put $K_\delta=\Ker(a_\delta)$.  The long exact sequence obtained by applying $G$ to the stable triangle identifies
\[
 K_\delta\simeq\Coker H(p_\delta).
\]
The index defect formula \cref{prop:index-defect} therefore gives
\begin{equation}\label{eq:index-from-stable-cone}
 \ind_T(Y_\delta)
 =\ind_T(M)+\ind_T(L)-\Phi([K_\delta]).
\end{equation}
On each stratum from (i) on which the rank is constant, the class $[K_\delta]\in K_0(\modu\Lambda)$ is constant. Only finitely many such classes occur.  Formula \eqref{eq:index-from-stable-cone} therefore proves (ii).  Notice that its values lie in the full group $K_0^{\mathrm{sp}}(\add T)$, so the projective--injective, or frozen, coordinates have not been discarded.
\end{proof}

For objects $Y,Y'\in\C$, write $Y\sim_TY'$ if
\[
 \ind_T(Y)=\ind_T(Y')
\]
and
\[
 \chi\bigl(\Gr_e(\F Y)\bigr)
 =\chi\bigl(\Gr_e(\F Y')\bigr)
 \quad\text{for every }e\in K_0(\modu\Lambda).
\]
Then $Y\sim_TY'$ implies $\X^T_Y=\X^T_{Y'}$ by \eqref{eq:character}.

For an object $Y$ which occurs as a middle term of an extension in $\E(L,M)$, let
\[
 \E(L,M)_{\langle Y\rangle}
 =\{\delta\in\E(L,M)\mid Y_\delta\sim_TY\}.
\]
The subset is stable under multiplication by nonzero scalars.  Let $\mathcal Y_{L,M}$ be a set of representatives of the nonempty equivalence classes.

\begin{proposition}\label{prop:stratification}
Under \cref{hyp:constructible-cones}, the decomposition
\[
 \E(L,M)=\bigsqcup_{Y\in\mathcal Y_{L,M}}
 \E(L,M)_{\langle Y\rangle}
\]
is a finite constructible stratification.
\end{proposition}

\begin{proof}
By \cref{prop:constructible-lift}(i), the module $\F Y_\delta$ varies in a constructible family of representations of $\Lambda$.  For any fixed $e$, the function
\[
 \delta\longmapsto\chi\bigl(\Gr_e(\F Y_\delta)\bigr)
\]
is constructible.  This is the standard pushforward of the characteristic function of the relative quiver Grassmannian.  Only finitely many $e$ occur, because exactness of \eqref{eq:four-term} bounds the dimension vector of every middle module in terms of the fixed end modules.  Part (ii) of \cref{prop:constructible-lift} gives a finite constructible stratification by the full index.  Intersecting these finitely many stratifications proves the assertion.
\end{proof}

If $S\subseteq\mathbb P\E(L,M)$ is constructible, the function
$[\delta]\mapsto\X^T_{Y_\delta}$ is integrable and
\[
 \int_S\X^T_{Y_\delta}\,d\chi
 =\sum_{Y\in\mathcal Y_{L,M}}
   \chi\bigl(S\cap\mathbb P\E(L,M)_{\langle Y\rangle}\bigr)\X^T_Y.
\]
\begin{remark}
The stable stratification is due to Palu \cite[Proposition~2.8]{Palu2012} and is restated by Keller, Plamondon and Qin \cite[Proposition~2.7]{KellerPlamondonQin}.  The proof above records why the full extriangulated index, including its frozen coordinates, remains constructible.
\end{remark}

\subsection{Incidence varieties}\label{sec:incidence}

Fix $L,M\in\C$ and a nonzero subspace $V\subseteq\E(L,M)$.  For $e,f\in K_0(\modu\Lambda)$, put
\[
 \mathcal B_{e,f}
 =\mathbb PV\times\Gr_e(\F L)\times\Gr_f(\F M).
\]
The product of the cluster characters may be written as
\begin{equation}\label{eq:lhs-expansion}
 \chi(\mathbb PV)\X^T_L\X^T_M
 =\bx^{\ind_T(L\oplus M)}
   \sum_{e,f}\chi(\mathcal B_{e,f})\bx^{-\Phi(e+f)}.
\end{equation}

Let $\mathcal W^V_{L,M}(e,f,g)$ be the set of pairs $([\delta],U)$ such that
\begin{itemize}[leftmargin=2em]
\item $0\neq\delta\in V$.
\item $U\subseteq\F Y_\delta$ is a $\Lambda$-submodule with $[U]=g$.
\item $[b_\delta(U)]=e$ and $[a_\delta^{-1}(U)]=f$.
\end{itemize}
Set
\[
 \mathcal W^V_{L,M}(e,f)=\bigsqcup_g\mathcal W^V_{L,M}(e,f,g).
\]
These conditions are well defined on the projective class $[\delta]$: changing the realisation gives an isomorphic module diagram, while multiplying $\delta$ by a nonzero scalar only rescales a structural map and does not change the relevant image or inverse image submodules.  By \cref{prop:constructible-lift}(i), these are constructible sets.  There is a constructible map
\[
\Psi_{e,f}\colon\mathcal W^V_{L,M}(e,f)\longrightarrow\mathcal B_{e,f},
 \hspace{2.5mm}
 ([\delta],U)\longmapsto
 ([\delta],b_\delta U,a_\delta^{-1}U).
\]
Let $\mathcal L^V_1(e,f)$ be its image and $
 \mathcal L^V_2(e,f)=\mathcal B_{e,f}\setminus\mathcal L^V_1(e,f).
$ Then we have the following result.

\begin{lemma}\label{lem:forward-fibres}
Every nonempty fibre of $\Psi_{e,f}$ is an affine space.  Consequently,
\[
 \chi\bigl(\mathcal L^V_1(e,f)\bigr)
 =\sum_g\chi\bigl(\mathcal W^V_{L,M}(e,f,g)\bigr).
\]
\end{lemma}

This is Palu \cite[Lemma~3.2]{Palu2012}, restricted to the subspace $V$ exactly as in Keller, Plamondon and Qin \cite[Lemma~2.13]{KellerPlamondonQin}.  The four-term sequence \eqref{eq:four-term} is the module diagram used there, and \cref{prop:constructible-lift}(i) supplies constructibility.

Let
\[
 \beta_{L,M}\colon\E(L,M)\times\E(M,L)\longrightarrow\kk
\]
be the nondegenerate pairing induced by \eqref{eq:2cy}.  Define
\[
 V^\perp=\{\eta\in\E(M,L)\mid \beta_{L,M}(v,\eta)=0
 \text{ for all }v\in V\}
\]
and
\[
 \mathcal R_V=\mathbb P\E(M,L)\setminus\mathbb P(V^\perp).
\]
For $0\neq\eta\in\E(M,L)$, choose a realisation
\[
 L\xto{j_\eta}Y'_\eta\xto{q_\eta}M\overset{\eta}{\etri}.
\]
Applying $G\Sigma$ gives maps
\[
 \F L\xto{a'_\eta}\F Y'_\eta\xto{b'_\eta}\F M.
\]
Define $\mathcal W^{\mathcal R}_{M,L}(f,e,g)$ as the set of pairs $([\eta],U')$ with $[\eta]\in\mathcal R_V$, $U'\subseteq\F Y'_\eta$, and
\[
 [U']=g,
\hspace{2.5mm}
 [b'_\eta U']=f,
\hspace{2.5mm}
 [(a'_\eta)^{-1}U']=e.
\]

The following is the geometric heart of the multiplication formula.  It is the $\E$-extension form of the correspondence proved by Palu \cite[Propositions~3.3 and~3.4]{Palu2012} and refined by Keller, Plamondon and Qin \cite[Lemmas~2.13 and~2.14]{KellerPlamondonQin}.

\begin{lemma}\label{lem:complement}
For every $e,f$, we have
\[
 \chi\bigl(\mathcal L^V_2(e,f)\bigr)
 =\sum_g\chi\bigl(\mathcal W^{\mathcal R}_{M,L}(f,e,g)\bigr).
\]
\end{lemma}

\begin{proof}
Let
\[
 \vartheta_{L,M}\colon\E(L,M)\xrightarrow{\sim}
 \Hom_{\A}(\pi L,\Sigma\pi M)
\]
be the canonical linear isomorphism attached to the Frobenius stable category, and define $\vartheta_{M,L}$ similarly.  These isomorphisms identify the $2$-Calabi--Yau pairing on $\E$ with the pairing on the stable category.  Hence $\vartheta_{L,M}(V)^\perp=\vartheta_{M,L}(V^\perp)$, and the induced projective isomorphism carries $\mathcal R_V$ onto the reverse locus $R$ of Keller, Plamondon and Qin \cite[Lemma~2.14]{KellerPlamondonQin}.

Wang, Wei and Zhang \cite[Lemma~3.1]{WangWeiZhangCC} show that a realisation of an $\E$-extension maps to the stable triangle determined by the corresponding morphism under $\vartheta_{L,M}$.  Applying $G\Sigma$ to that triangle gives exactly the four-term module diagram \eqref{eq:four-term}.  Thus, after using the constructible lifts supplied by \cref{hyp:constructible-cones}, the sets $\mathcal B_{e,f}$, $\mathcal L^V_2(e,f)$ and $\mathcal W^{\mathcal R}_{M,L}(f,e,g)$ are the pullbacks, under the above projective linear isomorphisms, of the corresponding sets of Keller, Plamondon and Qin \cite[Lemmas~2.13 and~2.14]{KellerPlamondonQin}.  Changing a realisation only changes the module diagram by a base change, and replacing an extension by a nonzero scalar changes only a structural map by a nonzero scalar. The image and inverse image submodules occurring in the incidence conditions are therefore unchanged up to the evident isomorphism.  Consequently the construction is independent of these choices.

The two projections in the proof of Keller, Plamondon and Qin \cite[Lemma~2.14]{KellerPlamondonQin} are surjective and have fibres of Euler characteristic one. This is the correspondence obtained from Palu \cite[Propositions~3.3 and~3.4]{Palu2012}. For each fixed pair $(e,f)$, its identity for Euler characteristics gives
\[
 \chi\bigl(\mathcal L^V_2(e,f)\bigr)
 =\sum_g\chi\bigl(\mathcal W^{\mathcal R}_{M,L}(f,e,g)\bigr).
\]
No projective--injective information is needed in this geometric correspondence. It enters again through the full index in \cref{lem:exponent}.
\end{proof}

\begin{lemma}\label{lem:exponent}
If $([\delta],U)\in\mathcal W^V_{L,M}(e,f,g)$, then
\begin{equation}\label{eq:forward-exponent}
 \ind_T(L\oplus M)-\Phi(e+f)
 =\ind_T(Y_\delta)-\Phi(g).
\end{equation}
Similarly, if $([\eta],U')\in\mathcal W^{\mathcal R}_{M,L}(f,e,g)$, then
\begin{equation}\label{eq:reverse-exponent}
 \ind_T(L\oplus M)-\Phi(e+f)
 =\ind_T(Y'_\eta)-\Phi(g).
\end{equation}
\end{lemma}

\begin{proof}
We only need to prove \eqref{eq:forward-exponent}.  Let
$\alpha=a_\delta|_{a_\delta^{-1}(U)}$, and exactness of
\eqref{eq:four-term} gives the exact sequence
\[
 0\longrightarrow\Ker\alpha\longrightarrow
 a_\delta^{-1}(U)\xto{\alpha}U
 \longrightarrow b_\delta(U)\longrightarrow0.
\]
Thus, in $K_0(\modu\Lambda)$, one obtains
\[
 [\Ker\alpha]=e+f-g.
\]
Since $0\in U$, one has $\Ker\alpha=\Ker a_\delta$. Exactness of the stable triangle identifies this kernel with the cokernel which appears when \cref{prop:index-defect} is applied to the $\E$-triangle \eqref{eq:realisation}.  Hence
\[
 \ind_T(M)-\ind_T(Y_\delta)+\ind_T(L)
 =\Phi(e+f-g).
\]
Since $\Phi$ is a group homomorphism, rearranging this equality gives \eqref{eq:forward-exponent}.
\end{proof}

\subsection{The refined multiplication formula}\label{sec:main}

We now prove the refined multiplication formula.

\begin{theorem}\label{thm:main}
Assume \cref{hyp:constructible-cones}.  Let $L,M\in\C$ and let
$0\neq V\subseteq\E(L,M)$ be a linear subspace.  Then
\begin{equation}\label{eq:main}
\begin{aligned}
 \chi(\mathbb PV)\X^T_L\X^T_M
 ={}&\sum_{Y\in\mathcal Y_{L,M}}
 \chi\bigl(\mathbb P(V\cap\E(L,M)_{\langle Y\rangle})\bigr)\X^T_Y\\
 &+\sum_{Y\in\mathcal Y_{M,L}}
 \chi\bigl(\mathcal R_V\cap
   \mathbb P\E(M,L)_{\langle Y\rangle}\bigr)\X^T_Y.
\end{aligned}
\end{equation}
\end{theorem}

\begin{proof}
Start with the expansion \eqref{eq:lhs-expansion}.  Since
\[
 \mathcal B_{e,f}=\mathcal L^V_1(e,f)\sqcup\mathcal L^V_2(e,f),
\]
additivity of Euler characteristic gives
\begin{align*}
 \chi(\mathbb PV)\X^T_L\X^T_M
 ={}&\bx^{\ind_T(L\oplus M)}
 \sum_{e,f}\chi\bigl(\mathcal L^V_1(e,f)\bigr)\bx^{-\Phi(e+f)}\\
 &+\bx^{\ind_T(L\oplus M)}
 \sum_{e,f}\chi\bigl(\mathcal L^V_2(e,f)\bigr)\bx^{-\Phi(e+f)}.
\end{align*}

By \cref{lem:forward-fibres}, the first term is
\[
 \sum_{e,f,g}\chi\bigl(\mathcal W^V_{L,M}(e,f,g)\bigr)
 \bx^{\ind_T(L\oplus M)-\Phi(e+f)}.
\]
Decompose each incidence variety according to the strata of middle terms of \cref{prop:stratification}.  On the stratum indexed by $Y$, \eqref{eq:forward-exponent} replaces the monomial by
\[
 \bx^{\ind_T(Y)-\Phi(g)}.
\]
For fixed $[\delta]$ and $U\subseteq\F Y_\delta$, the classes $e=[b_\delta(U)]$ and $f=[a_\delta^{-1}(U)]$ are uniquely determined.  Hence summing over $e$ and $f$ simply forgets these two labels and recovers the quiver Grassmannian of the middle module.  The first term is therefore
\[
 \sum_{Y\in\mathcal Y_{L,M}}
 \chi\bigl(\mathbb P(V\cap\E(L,M)_{\langle Y\rangle})\bigr)
 \bx^{\ind_T(Y)}
 \sum_g\chi\bigl(\Gr_g(\F Y)\bigr)\bx^{-\Phi(g)}.
\]
By \eqref{eq:character}, this is the first sum on the right side of \eqref{eq:main}.

For the second term, \cref{lem:complement} replaces
$\mathcal L^V_2(e,f)$ by the incidence varieties of reverse extensions in $\mathcal R_V$.  The reverse exponent identity \eqref{eq:reverse-exponent} then gives
\[
 \sum_{Y\in\mathcal Y_{M,L}}
 \chi\bigl(\mathcal R_V\cap
   \mathbb P\E(M,L)_{\langle Y\rangle}\bigr)
 \bx^{\ind_T(Y)}
 \sum_g\chi\bigl(\Gr_g(\F Y)\bigr)\bx^{-\Phi(g)}.
\]
This is the second sum in \eqref{eq:main}.
\end{proof}

\begin{remark}\label{rem:no-extra-term}
The proof of \cref{thm:main} uses constructibility only for the cylinders attached to the two fixed directions $(L,M)$ and $(M,L)$.  Thus the global \cref{hyp:constructible-cones} may be weakened, for this theorem, to the corresponding pairwise constructibility condition.  We keep the global hypothesis in the main statements because it is the standard form used in the applications below.

There is no universal extra summand attached to the projective--injective objects in \eqref{eq:main}.  Instead, every actual middle term in $\C$ is used.  If
$Y=Y_0\oplus P$ with $P$ projective--injective, then
\[
 \X^T_Y=\X^T_{Y_0}\X^T_P,
\]
so the frozen monomial is already present in its correct stratum.
\end{remark}

Taking $V=\E(L,M)$, \cref{thm:main} gives the following form of Palu's formula with coefficients retained.

\begin{corollary}\label{cor:Palu}
Assume \cref{hyp:constructible-cones}.  Then
\begin{align*}
 \chi\bigl(\mathbb P\E(L,M)\bigr)\X^T_L\X^T_M
={}&\int_{\mathbb P\E(L,M)}\X^T_{Y_\delta}\,d\chi
 +\int_{\mathbb P\E(M,L)}\X^T_{Y'_\eta}\,d\chi.
\end{align*}
\end{corollary}

\begin{proof}
If $\E(L,M)=0$, then $\E(M,L)=0$ by $2$-Calabi--Yau duality, and both sides vanish by the convention $\mathbb P0=\varnothing$.  Otherwise take $V=\E(L,M)$ in \cref{thm:main}.  Nondegeneracy of the pairing gives $V^\perp=0$, and hence $\mathcal R_V=\mathbb P\E(M,L)$.
\end{proof}

Note that $\chi(\mathbb P U)=\dim_{\kk}U$ for a nonzero complex vector space $U$.  The one-dimensional exchange relation of Wang, Wei and Zhang \cite[Theorem~4.4]{WangWeiZhangCC} is now an immediate specialisation of the main theorem.

\begin{corollary}\label{cor:exchange}
Suppose that $\dim_{\kk}\E(L,M)=1$.  Let
\[
 M\longrightarrow Y\longrightarrow L\etri,
\hspace{2.5mm}
 L\longrightarrow Y'\longrightarrow M\etri
\]
be the two nonsplit $\E$-triangles.  Then
\[
 \X^T_L\X^T_M=\X^T_Y+\X^T_{Y'}.
\]
\end{corollary}

\begin{proof}
Take $V=\E(L,M)$ in \cref{thm:main}. Both projective extension spaces consist of one point, and $V^\perp=0$.
\end{proof}

Set $x_{r+1}=\cdots=x_n=1$.  By the comparison theorem of Wang, Wei and Zhang \cite[Proposition~4.5]{WangWeiZhangCC}, the specialised character of $M$ is Palu's character of $\Sigma\pi M$ with respect to $T'$.  Projective--injective summands disappear, and shifting the stable triangles shows that \cref{thm:main} becomes the refined formula of Keller, Plamondon and Qin in $\A$.

\begin{corollary}\label{cor:stable}
After frozen specialisation, \eqref{eq:main} is the refined multiplication formula for the Palu character in the stable $2$-Calabi--Yau triangulated category $\A$.
\end{corollary}

\begin{lemma}\label{lem:stable-lift}
Let $X,Y$ be indecomposable non-projective objects of $\C$.  If $\pi X\simeq\pi Y$ in $\A$, then $X\simeq Y$ in $\C$.  In particular, every automorphism of $\pi X$ in the stable category admits a lift which is an automorphism of $X$.
\end{lemma}

\begin{proof}
Choose morphisms $f\colon X\to Y$ and $g\colon Y\to X$ whose stable classes are mutually inverse.  Then $gf-1_X$ and $fg-1_Y$ factor through projective--injective objects.  Such an endomorphism of an indecomposable non-projective object cannot be invertible: if $u\colon X\to X$ factors through $P\in\Pcal$ and is invertible, then $1_X=u^{-1}u$ factors through $P$, so $X$ is a direct summand of $P$ and hence projective, a contradiction.  Since $\End_{\C}(X)$ and $\End_{\C}(Y)$ are local, these endomorphisms lie in the respective Jacobson radicals.  Consequently $gf$ and $fg$ are invertible, and therefore $f$ is an isomorphism.

For the last assertion, choose a lift $a\colon X\to X$ of a stable automorphism and a lift $b$ of its stable inverse.  The same argument shows that $ba$ and $ab$ are invertible, hence so is $a$.
\end{proof}

For $\delta\in\E(C,A)$, a morphism $a\colon A\to A'$ induces
\[
 a_*\delta:=\E(C,a)(\delta)\in\E(C,A'),
\]
and a morphism $c\colon C'\to C$ induces
\[
 c^*\delta:=\E(c,A)(\delta)\in\E(C',A).
\]
A nonzero extension $\delta\in\E(C,A)$ is \emph{almost split} if
$a_*\delta=0$ for every morphism $a\colon A\to A'$ that is not a section and
$c^*\delta=0$ for every morphism $c\colon C'\to C$ that is not a retraction.  An
$\E$-triangle realising such an extension is called an Auslander--Reiten
$\E$-triangle.  For an indecomposable non-projective object $Z$, its first
term is denoted by $\tau Z$.  These notions and the existence criterion via
Auslander--Reiten--Serre duality are due to Iyama, Nakaoka and Palu \cite[Definition~3.1 and
Theorem~3.6]{IyamaNakaokaPalu}.  In the present setting the $2$-Calabi--Yau
duality supplies that duality.  On the stable category one has
$\pi(\tau Z)\simeq\Sigma\pi Z$.

Choose the representative $\tau Z$ without projective--injective direct summands and fix an isomorphism
\[
 \xi\colon\pi(\tau Z)\xrightarrow{\sim}\Sigma\pi Z.
\]
Under the canonical isomorphism
\[
 \E(\tau Z,Z)\xrightarrow{\sim}
 \Hom_{\A}(\pi(\tau Z),\Sigma\pi Z),
\]
let $\delta_0$ be the extension corresponding to $\xi$, and choose a realisation
\[
 Z\xrightarrow{i_Z}I_Z\longrightarrow\tau Z
 \overset{\delta_0}{\etri}.
\]
Its image in the stable category has connecting morphism $\xi$, hence is isomorphic to the standard suspension triangle
\[
 \pi Z\longrightarrow0\longrightarrow\pi(\tau Z)
       \xrightarrow{\xi}\Sigma\pi Z.
\]
Therefore $\pi I_Z=0$, so $I_Z\in\Pcal$.  Since $i_Z$ is an inflation and every object of $\Pcal$ is injective, $i_Z$ is a left $\Pcal$-approximation.

It is left minimal.  Indeed, suppose that $s\in\End_{\C}(I_Z)$ satisfies $s i_Z=i_Z$.  Axiom \textup{(ET3)} completes $(1_Z,s)$ to a morphism of the above $\E$-triangle to itself, with some endomorphism $t$ of $\tau Z$.  Passing to the stable category and using the last square gives
\[
 \xi\,\pi(t)=\xi,
\]
so $\pi(t)=1_{\pi(\tau Z)}$.  Thus $t-1_{\tau Z}$ factors through a projective--injective object. Since $\tau Z$ is indecomposable and non-projective, the argument for local rings in \cref{lem:stable-lift} shows that $t$ is an automorphism.  In a morphism of $\E$-triangles, if the maps on the first and third terms are isomorphisms, then so is the map on the middle term. Hence $s$ is an automorphism.  Thus $i_Z$ is a minimal left $\Pcal$-approximation.

We therefore fix an Auslander--Reiten $\E$-triangle and this minimal companion cosyzygy $\E$-triangle
\begin{equation}\label{eq:AR-and-cosy}
 \tau Z\longrightarrow Y\longrightarrow Z\overset{\delta_{\mathrm{AR}}}{\etri},
 \hspace{2.5mm}
 Z\longrightarrow I_Z\longrightarrow\tau Z\overset{\delta_0}{\etri},
 \hspace{2.5mm}I_Z\in\Pcal.
\end{equation}
Notice that $\delta_0\ne0$ whenever $Z\ne0$ in $\A$, because its stable connecting morphism $\xi$ is an isomorphism.

\begin{lemma}[Keller--Plamondon--Qin \cite{KellerPlamondonQin}]\label{known:AR-locus}
Let $V=\kk\delta_{\mathrm{AR}}\subseteq\E(Z,\tau Z)$, and let
\[
 \xi\colon\pi(\tau Z)\xrightarrow{\sim}\Sigma\pi Z
\]
be the connecting isomorphism of the companion cosyzygy triangle.  Via
\[
 \E(\tau Z,Z)
 \simeq\Hom_{\A}(\pi(\tau Z),\Sigma\pi Z)
 \xrightarrow{\ \xi^{-1}\circ-\ }
 \End_{\A}(\pi(\tau Z)),
\]
one has
\[
 \mathcal R_V
 \cong
 \mathbb P\End_{\A}(\pi(\tau Z))
 \setminus
 \mathbb P\operatorname{rad}\End_{\A}(\pi(\tau Z)).
\]
This locus is an affine space, hence has Euler characteristic one, and every class in it has an invertible connecting morphism $\pi(\tau Z)\to\Sigma\pi Z$.
\end{lemma}

The statement and its radical orthogonality proof are contained in Keller, Plamondon and Qin \cite[Section~4.3, proof of Theorem~4.8, pp.~14--15]{KellerPlamondonQin}.

\begin{lemma}\label{lem:AR-loci}
For every $[\eta]\in\mathcal R_V$, the middle term of a realisation
\[
 Z\longrightarrow Y'_\eta\longrightarrow\tau Z\overset{\eta}{\etri}
\]
is isomorphic to $I_Z$.  Consequently,
\[
 \int_{\mathcal R_V}\X^T_{Y'_\eta}\,d\chi=\X^T_{I_Z}.
\]
\end{lemma}

\begin{proof}
This is the step which retains the coefficient information and is absent after passage to the stable category.  Let
\[
 \xi\colon\pi(\tau Z)\xrightarrow{\sim}\Sigma\pi Z
\]
be the connecting isomorphism of the second $\E$-triangle in \eqref{eq:AR-and-cosy}.  By \cref{known:AR-locus}, the connecting morphism $\overline\eta\colon\pi(\tau Z)\to\Sigma\pi Z$ corresponding to $\eta$ is invertible.  Hence
\[
 \alpha=\xi^{-1}\overline\eta\in\Aut_{\A}(\pi(\tau Z)).
\]
By \cref{lem:stable-lift}, choose an actual automorphism $a\colon\tau Z\to\tau Z$ whose stable class is $\alpha$.  Pulling the companion cosyzygy $\E$-triangle back along $a$ produces the extension $a^*\delta_0$, where $\delta_0\in\E(\tau Z,Z)$ is its original extension class, and its stable connecting morphism is $\xi\,\pi(a)=\xi\alpha=\overline\eta$.  The natural isomorphism
\[
 \E(\tau Z,Z)\simeq\Hom_{\A}(\pi(\tau Z),\Sigma\pi Z)
\]
is injective, so $a^*\delta_0=\eta$.  Since $a$ is an automorphism, the pullback middle term is isomorphic to $I_Z$.  Thus every realisation of $\eta$ has middle term isomorphic to $I_Z$, and therefore $Y'_\eta\cong I_Z$.  The character is constant on $\mathcal R_V$, and the last assertion follows from $\chi(\mathcal R_V)=1$.
\end{proof}

\begin{corollary}[Explicit coefficient term in an Auslander--Reiten mesh]\label{cor:AR}
Assume \cref{hyp:constructible-cones}.  For the triangles in \eqref{eq:AR-and-cosy}, set
\[
 \varepsilon_Z
 :=\ind_T Z+\ind_T(\tau Z)-\Phi([H(\tau Z)])
 \in K_0^{\mathrm{sp}}(\add T).
\]
Then
\begin{equation}\label{eq:AR-frozen-class}
 \varepsilon_Z=[I_Z]
 \in \bigoplus_{i=r+1}^{n}\mathbb Z_{\ge0}[T_i].
\end{equation}
In particular,
\begin{equation}\label{eq:AR-explicit}
 \X^T_{\tau Z}\X^T_Z
 =\X^T_Y+\bx^{\varepsilon_Z}.
\end{equation}
Equivalently, if $I_Z\cong\bigoplus_{i=r+1}^{n}T_i^{m_i(Z)}$, then
\[
 \varepsilon_Z=\sum_{i=r+1}^{n}m_i(Z)[T_i],
 \hspace{2.5mm}
 \bx^{\varepsilon_Z}=\prod_{i=r+1}^{n}x_i^{m_i(Z)}.
\]
\end{corollary}

\begin{proof}
Apply \cref{thm:main} with $L=Z$, $M=\tau Z$ and
$V=\kk\delta_{\mathrm{AR}}$.  Since $\chi(\mathbb PV)=1$, \cref{lem:AR-loci} gives
\[
 \X^T_{\tau Z}\X^T_Z=\X^T_Y+\X^T_{I_Z}.
\]
It remains to identify the second term intrinsically.  Apply the index defect formula \cref{prop:index-defect} to the companion cosyzygy $\E$-triangle
\[
 Z\longrightarrow I_Z\xto{g}\tau Z\dashrightarrow.
\]
Since $I_Z$ is projective--injective, $\pi I_Z=0$ and hence $H(I_Z)=0$.  Therefore
\[
 \Coker H(g)=H(\tau Z),
\]
and \cref{prop:index-defect} yields
\[
 \ind_T Z-\ind_T I_Z+\ind_T(\tau Z)
 =\Phi([H(\tau Z)]).
\]
Every projective--injective object belongs to $\add T$, and for an object of $\add T$ its index is its split Grothendieck class.  Thus
\[
 \ind_T I_Z=[I_Z],
\]
which proves \eqref{eq:AR-frozen-class}.  Since $I_Z$ is a direct sum of the projective--injective summands $T_{r+1},\ldots,T_n$, the right side of \eqref{eq:AR-frozen-class} lies in the indicated positive frozen cone.  Finally, $\F I_Z=0$, so the cluster character formula \eqref{eq:character} gives
\[
 \X^T_{I_Z}=\bx^{[I_Z]}=\bx^{\varepsilon_Z}.
\]
Substituting this into the preceding mesh relation proves \eqref{eq:AR-explicit}.
\end{proof}

After specialising the frozen variables to $1$, the monomial $\bx^{\varepsilon_Z}$ becomes $1$, and \eqref{eq:AR-explicit} recovers the Auslander--Reiten multiplication formula of Dominguez--Geiss \cite{DominguezGeiss}.

\begin{remark}
The extension space may have dimension greater than one, since
\[
 \E(Z,\tau Z)\cong D\End_{\A}(\pi Z).
\]
\end{remark}

\section{Higgs categories and principal coefficients}\label{Higgs}

\subsection{Jacobi-finite Higgs categories}\label{subsec:general-higgs}
We now apply the coefficient lift to a genuinely extriangulated source of
examples.  We recall the notation from Wu \cite{WuRelative}.  An
\emph{ice quiver with potential} is a finite quiver $Q$, a frozen subquiver
$F$, and a potential $W$ on $Q$.  The relative Ginzburg dg algebra is denoted
by $\Gamma_{\mathrm{rel}}(Q,F,W)$, and its zeroth cohomology
\[
 J(Q,F,W)=H^0\Gamma_{\mathrm{rel}}(Q,F,W)
 =\kk Q/\langle\partial_aW\mid a\in Q_1\setminus F_1\rangle
\]
is the relative Jacobian algebra.  The triple is \emph{Jacobi-finite} when this
algebra is finite-dimensional.  If $e=\sum_{i\in F_0}e_i$, the relative
cluster category is the Verdier quotient specified by Wu
\cite[Definition~7.9]{WuRelative}.  Writing $\mathcal P=\add(e\Gamma_{\mathrm{rel}})$, its relative fundamental domain is
\[
 \mathcal F^{\mathrm{rel}}=
 \{\operatorname{Cone}(f:X_1\to X_0)\mid
 X_0,X_1\in\add\Gamma_{\mathrm{rel}},
 \Hom(f,I)\text{ is surjective for every }I\in\mathcal P\}.
\]  The \emph{Higgs category}
$\mathcal H(Q,F,W)$ is the image of that fundamental domain in the quotient.

We distinguish the dg algebra from its image as an object: write
$\boldsymbol\Gamma$ for the image of the free dg module
$\Gamma_{\mathrm{rel}}$ in $\mathcal H(Q,F,W)$, and set
$\Gamma_i=e_i\boldsymbol\Gamma$.  Wu \cite[Theorem~7.10]{WuRelative} proves that the Higgs category is
Hom-finite Frobenius $2$-Calabi--Yau extriangulated, that
$\boldsymbol\Gamma=\bigoplus_{i\in Q_0}\Gamma_i$ is cluster tilting, and that
its projective--injective subcategory is
$\add(e\boldsymbol\Gamma)=\add(\bigoplus_{i\in F_0}\Gamma_i)$.  No exact structure on $\mathcal H(Q,F,W)$ is used. Fraser, Keller and Wu \cite[Theorem~5.4]{FraserKellerWu} give another construction of a Higgs cluster character. Here all multiplication statements refer to the Wang--Wei--Zhang character \eqref{eq:character}.

\begin{proposition}\label{prop:Higgs-constructible}
Let $(Q,F,W)$ be a Jacobi-finite ice quiver with potential.  Then the stable category $\underline{\mathcal H(Q,F,W)}$ has constructible cones with respect to the image of the non-projective part of $\boldsymbol\Gamma$.  In particular, \cref{hyp:constructible-cones} holds for the Higgs category.
\end{proposition}

\begin{proof}
Delete the frozen vertices and all incident arrows, and denote the resulting
quiver with potential by $(\overline Q,\overline W)$.  If $e$ is the sum of the
frozen idempotents, then Wu \cite[Definition~7.6 and Proposition~7.8]{WuRelative} gives,
after taking zeroth cohomology of the dg quotient,
\[
 J(\overline Q,\overline W)
 \cong J(Q,F,W)/J(Q,F,W)eJ(Q,F,W).
\]
Thus $(\overline Q,\overline W)$ is Jacobi-finite.  Let
$\mathcal C(\overline Q,\overline W)$ denote Amiot's generalized cluster category associated with $(\overline Q,\overline W)$.  Wu \cite[Theorem~7.10]{WuRelative} gives a triangle equivalence
\[
 \underline{\mathcal H(Q,F,W)}
 \simeq
 \mathcal C(\overline Q,\overline W)
\]
which sends the non-projective part of $\boldsymbol\Gamma$ to the canonical cluster tilting object of the generalized cluster category.  Generalized cluster categories associated with Jacobi-finite quivers with potential have constructible cones by Palu \cite[Sections~2.4 and~2.5]{Palu2012}.  Keller, Plamondon and Qin \cite[Example~2.9]{KellerPlamondonQin} give an explicit summary.  Constructible cones are invariant under triangle equivalence.  Hence the stable category has constructible cones with respect to the image of the non-projective part of $\boldsymbol\Gamma$, which is exactly \cref{hyp:constructible-cones} in this setting.
\end{proof}

\begin{corollary}\label{cor:Higgs-main}
Let $(Q,F,W)$ be a Jacobi-finite ice quiver with potential.  For all $L,M\in\mathcal H(Q,F,W)$ and every nonzero subspace
\[
 V\subseteq\E_{\mathcal H(Q,F,W)}(L,M),
\]
the refined multiplication formula \eqref{eq:main} holds for the Wang--Wei--Zhang character associated with $\boldsymbol\Gamma(Q,F,W)$.
\end{corollary}

\begin{proof}
Combine \cref{prop:Higgs-constructible} with \cref{thm:main}.
\end{proof}

\begin{corollary}\label{cor:frozen-mesh}
Let $Z$ be indecomposable and non-projective.  Write its Auslander--Reiten
$\E$-triangle and a minimal companion cosyzygy triangle as
\[
 \tau Z\longrightarrow\bigoplus_{j=1}^{s}Y_j\longrightarrow Z
   \dashrightarrow,
 \hspace{2.5mm}
 Z\longrightarrow\bigoplus_{i\in F_0}\Gamma_i^{m_i(Z)}
   \longrightarrow\tau Z\dashrightarrow.
\]
Then the frozen multiplicity vector is determined by
\begin{equation}\label{eq:Higgs-frozen-class}
 \sum_{i\in F_0}m_i(Z)[\Gamma_i]
 =\ind_{\boldsymbol\Gamma} Z
  +\ind_{\boldsymbol\Gamma}(\tau Z)
  -\Phi([H(\tau Z)]).
\end{equation}
Consequently,
\begin{equation}\label{eq:frozen-mesh}
 \X^{\boldsymbol\Gamma}_{\tau Z}\X^{\boldsymbol\Gamma}_{Z}
 =\prod_{j=1}^{s}\X^{\boldsymbol\Gamma}_{Y_j}
  +\prod_{i\in F_0}x_i^{m_i(Z)}.
\end{equation}
\end{corollary}

\begin{proof}
Apply \cref{cor:AR} in the Higgs category.
\end{proof}

\subsection{Acyclic principal coefficients and frozen multiplicities}\label{subsec:principal}

Let $Q$ be a finite nonempty acyclic quiver and put $A=\kk Q$.  We work with right modules and write $P_i=e_iA$ for the indecomposable projective at $i\in Q_0$.  Form the principal ice quiver $\widehat Q$ by adjoining, for every mutable vertex $i\in Q_0$, a frozen vertex $i^+$ and a framing arrow
\[
 c_i\colon i^+\longrightarrow i.
\]
Let $F$ be the full subquiver on the frozen vertices and take $W=0$.  Paths are composed from right to left, as in Wu \cite[Definition~7.5]{WuRelative}. Thus $e_uAe_v$ consists of paths from $v$ to $u$. The algebra acting on $G$ is $\End(T')$, without an opposite. For the seed convention we use the Gabriel quiver of $\End_{\mathcal H_Q}(\boldsymbol\Gamma)^{\mathrm{op}}$. Its mutable part is $Q^{\mathrm{op}}$ and its framing arrows are $i\to i^+$. Accordingly this model has principal coefficients for $Q^{\mathrm{op}}$ in that seed convention. Since $Q$ is arbitrary acyclic, this labels the full acyclic family. In comparisons with formulas for left modules, $\modu A$ is identified with $\kk Q^{\mathrm{op}}\text{-}\modu$.  Put
\[
 \Gamma_{\mathrm{rel}}=\Gamma_{\mathrm{rel}}(\widehat Q,F,0),
 \hspace{2.5mm}
 \mathcal H_Q=\mathcal H(\widehat Q,F,0).
\]
For a mutable vertex $j$, write $\widetilde\Gamma_j=e_j\Gamma_{\mathrm{rel}}$ for the corresponding right dg projective and $\Gamma_j$ for its image in $\mathcal H_Q$.  For a frozen vertex $i^+$, put
\[
 \widetilde I_i=e_{i^+}\Gamma_{\mathrm{rel}},
  \hspace{2.5mm}
 I_i=\Gamma_{i^+}.
\]
By Wu \cite[Theorem~7.10]{WuRelative}, $\mathcal H_Q$ is a Hom-finite Frobenius $2$-Calabi--Yau extriangulated category,
\[
 \proj\mathcal H_Q=\inj\mathcal H_Q
 =\add\Bigl(\bigoplus_{i\in Q_0}I_i\Bigr),
\]
and its stable category is triangle equivalent to the acyclic cluster category
\[
 \underline{\mathcal H_Q}\simeq\mathcal C_Q,
 \hspace{2.5mm}
 \mathcal C_Q=D^b(\modu A)/(\tau_D^{-1}\Sigma),
\]
where Buan, Marsh, Reineke, Reiten and Todorov \cite{BMRRT} give the description of the orbit category $\mathcal C_Q$.

\subsubsection{The Higgs lift of a hereditary module}
For any finite-dimensional right $A$-module $M$, choose its minimal projective resolution
\begin{equation}\label{eq:principal-projres}
 0\longrightarrow P_1\xrightarrow{f}P_0\longrightarrow M\longrightarrow0.
\end{equation}
Lift the mutable projectives in $P_0,P_1$ to direct sums $\widetilde P_0,\widetilde P_1$ of the dg projectives $\widetilde\Gamma_j$ and choose the induced lift $\widetilde f\colon\widetilde P_1\to\widetilde P_0$.  Set $\widetilde Z_M=\operatorname{Cone}(\widetilde f)$.  There is no path of degree zero from a mutable vertex to a frozen vertex, hence
\[
 \Hom_{\operatorname{per}\Gamma_{\mathrm{rel}}}(\widetilde P_r,\widetilde I_i)=0
  \hspace{2.5mm}(r=0,1).
\]
Thus the surjectivity condition in Wu's relative fundamental domain is automatic, and we denote by $Z_M\in\mathcal H_Q$ the image of $\widetilde Z_M$.

\begin{lemma}\label{lem:principal-lift-minimal}
Under the equivalence $\underline{\mathcal H_Q}\simeq\mathcal C_Q$, the stable image of $Z_M$ is isomorphic to the image of $M$ in $\mathcal C_Q$.  Moreover, $Z_M$ has no nonzero projective--injective direct summand.  In particular, if $M$ is indecomposable, then $Z_M$ is indecomposable and non-projective in $\mathcal H_Q$. For $M=P_j$, one has $Z_M=\Gamma_j$.
\end{lemma}
\begin{proof}
The distinguished triangle
\[
 \widetilde P_1\longrightarrow\widetilde P_0\longrightarrow
 \widetilde Z_M\longrightarrow\Sigma\widetilde P_1
\]
passes, under the stable equivalence, to the triangle determined by the projective resolution \eqref{eq:principal-projres}.  Hence the stable image of $Z_M$ is the image of $M$ in $\mathcal C_Q$.

Let $P_r^{\boldsymbol\Gamma}$ be the image of $\widetilde P_r$ in $\mathcal H_Q$.  The same distinguished triangle induces an $\E$-triangle
\[
 P_1^{\boldsymbol\Gamma}\longrightarrow P_0^{\boldsymbol\Gamma}
 \longrightarrow Z_M\dashrightarrow,
 \hspace{2.5mm}
 P_r^{\boldsymbol\Gamma}\in\add\!\Bigl(\bigoplus_{j\in Q_0}\Gamma_j\Bigr).
\]
Suppose that $Z_M\simeq Z'\oplus I$ with $0\ne I$ projective--injective, and write
\[
 g\colon P_0^{\boldsymbol\Gamma}\longrightarrow Z'\oplus I
\]
for the deflation in the displayed $\E$-triangle.  Let $\iota\colon I\to Z'\oplus I$ and $\rho\colon Z'\oplus I\to I$ be the canonical inclusion and projection.  Since $I$ is projective and $g$ is a deflation, $\iota$ lifts through $g$: there is a morphism $s\colon I\to P_0^{\boldsymbol\Gamma}$ with $gs=\iota$.  Setting $r=\rho g$ gives
\[
 rs=\rho gs=\rho\iota=1_I.
\]
Thus $I$ is a direct summand of $P_0^{\boldsymbol\Gamma}$.  This is impossible, because $P_0^{\boldsymbol\Gamma}$ is a direct sum of mutable summands $\Gamma_j$, whereas every indecomposable projective--injective is one of the frozen summands $I_i=\Gamma_{i^+}$ of the basic cluster tilting object $\boldsymbol\Gamma$.

If $M$ is indecomposable, Buan, Marsh, Reineke, Reiten and Todorov \cite{BMRRT} show that the standard fundamental domain for the hereditary cluster category identifies indecomposable $A$-modules with indecomposable objects of $\mathcal C_Q$.  If $Z_M$ decomposed nontrivially, then its stable image would decompose. Hence one summand would have zero stable image and would therefore be projective--injective, contradicting the preceding paragraph.  Since its stable image is nonzero, $Z_M$ is non-projective.
\end{proof}

\subsubsection{The degree \texorpdfstring{$-1$}{-1} frozen corner}
We use Wu's convention for right modules: an element of $e_u\Gamma_{\mathrm{rel}}e_v$ is represented by a path beginning at $v$ and ending at $u$.  Thus, if $p\colon j\rightsquigarrow i$ is an ordinary path in $Q$, then $c_i^*p$ means that one first follows $p$ and then the reverse degree $-1$ arrow $c_i^*\colon i\to i^+$.

\begin{lemma}\label{lem:corner}
For mutable vertices $i,j\in Q_0$ there is a natural isomorphism
\[
 H^{-1}(e_{i^+}\Gamma_{\mathrm{rel}}e_j)
 \simeq e_iAe_j
 \simeq\Hom_A(P_j,P_i),
\]
which sends the class of $c_i^*p$ to the ordinary path $p\colon j\rightsquigarrow i$.
\end{lemma}
\begin{proof}
By Wu \cite[Definition~7.5]{WuRelative}, each arrow of $\widehat Q$ which is not frozen has a reverse arrow of degree $-1$, and each mutable vertex has a loop of degree $-2$.  Since $W=0$, $d(a^*)=\partial_aW=0$ for every reverse arrow, in particular $d(c_i^*)=0$.

A degree $-1$ path from mutable $j$ to frozen $i^+$ contains exactly one reverse arrow.  Since no arrow of degree zero enters a frozen vertex, that reverse arrow must be $c_i^*$ and must be the final arrow.  Hence
\[
 (e_{i^+}\Gamma_{\mathrm{rel}}e_j)^{-1}
 =\bigoplus_{p:j\rightsquigarrow i}\kk\,c_i^*p,
\]
and all these elements are closed.

A degree $-2$ path in this corner either contains two degree $-1$ reverse arrows or one degree $-2$ loop.  The latter case is impossible: the degree $-2$ loops exist only at mutable vertices, and after such a loop one would need a path of degree zero from a mutable vertex to the frozen endpoint $i^+$.  Thus every degree $-2$ path in the corner contains two degree $-1$ reverse arrows.  Their differentials vanish, so the Leibniz rule gives
\[
 d\bigl((e_{i^+}\Gamma_{\mathrm{rel}}e_j)^{-2}\bigr)=0.
\]
Hence there are no degree $-2$ boundaries in degree $-1$, and
\[
 H^{-1}(e_{i^+}\Gamma_{\mathrm{rel}}e_j)
 =\bigoplus_{p:j\rightsquigarrow i}\kk[c_i^*p]
 \simeq e_iAe_j.
\]
For right modules, $\Hom_A(e_jA,e_iA)\simeq e_iAe_j$.
\end{proof}

\begin{remark}\label{rem:corner-naturality}
If $u\colon j\rightsquigarrow\ell$ and $v\colon\ell\rightsquigarrow i$ are ordinary paths, then $(c_i^*v)u=c_i^*(vu)$.  Thus \cref{lem:corner} is compatible with precomposition and hence with matrices of paths representing $f$ in \eqref{eq:principal-projres}.
\end{remark}
\subsubsection{Frozen multiplicities}
\begin{proposition}\label{prop:principal-multiplicity}
Let $M$ be an indecomposable non-projective right $A$-module. Write the minimal companion cosyzygy of $Z_M$ as
\[
 Z_M\longrightarrow I_{Z_M}\longrightarrow\tau Z_M\dashrightarrow,
  \hspace{2.5mm}I_{Z_M}\simeq\bigoplus_{i\in Q_0}I_i^{m_i(Z_M)}.
\]
Then
\[
 m_i(Z_M)=\dim_{\kk}\Ext_A^1(M,P_i)=\dim_{\kk}(\tau_A M)_i.
\]
\end{proposition}
\begin{proof}
There are no nontrivial paths of degree zero between frozen vertices.  By Wu \cite[Proposition~5.20]{WuRelative}, the relative quotient is fully faithful on the relative fundamental domain, and hence
\[
 \End_{\mathcal H_Q}(I_i)=\kk,
  \hspace{2.5mm} \Hom_{\mathcal H_Q}(I_j,I_i)=0 \hspace{2.5mm}(j\ne i).
\]
Let $u:Z_M\to I_{Z_M}$ be the minimal left projective--injective approximation.  Since $u$ is a left $\add(\bigoplus_i I_i)$-approximation, composition with $u$ gives a surjection
\[
 u^*: \Hom_{\mathcal H_Q}(I_{Z_M},I_i)\longrightarrow
 \Hom_{\mathcal H_Q}(Z_M,I_i).
\]
The source is naturally $\kk^{m_i(Z_M)}$.  The map $u^*$ is also injective.  Indeed, if $0\ne p\in\Hom_{\mathcal H_Q}(I_{Z_M},I_i)$ satisfied $pu=0$, then after a change of basis among the $I_i$-summands we could arrange that one coordinate projection vanishes on $u$.  The endomorphism of $I_{Z_M}$ which kills that copy of $I_i$ and is the identity on all other summands would then satisfy $su=u$ without being invertible, contradicting left minimality.  Therefore $u^*$ is an isomorphism and
\begin{equation}\label{eq:principal-m-hom}
 m_i(Z_M)=\dim_{\kk}\Hom_{\mathcal H_Q}(Z_M,I_i).
\end{equation}
Apply $\Hom(-,\widetilde I_i)$ in $\operatorname{per}\Gamma_{\mathrm{rel}}$ to the triangle defining $\widetilde Z_M$.  Since $\Hom(\widetilde P_r,\widetilde I_i)=0$, the relevant part of the long exact sequence is
\[
 \Hom(\widetilde P_0,\Sigma^{-1}\widetilde I_i)
 \longrightarrow \Hom(\widetilde P_1,\Sigma^{-1}\widetilde I_i)
 \longrightarrow \Hom(\widetilde Z_M,\widetilde I_i)\longrightarrow0.
\]
The preceding construction shows $\widetilde Z_M\in\mathcal F^{\mathrm{rel}}$, while $\widetilde I_i=e_{i^+}\Gamma_{\mathrm{rel}}\in\add\Gamma_{\mathrm{rel}}\subset\mathcal F^{\mathrm{rel}}$.  Hence the full faithfulness result of Wu \cite[Proposition~5.20]{WuRelative} identifies the last term with $\Hom_{\mathcal H_Q}(Z_M,I_i)$.  Since all dg modules are right modules,
\[
 \Hom_{\operatorname{per}\Gamma_{\mathrm{rel}}}(\widetilde\Gamma_j,\Sigma^{-1}\widetilde I_i)
 \simeq H^{-1}(e_{i^+}\Gamma_{\mathrm{rel}}e_j).
\]
By \cref{lem:corner,rem:corner-naturality}, the displayed cokernel is exactly the cokernel of
\[
 \Hom_A(P_0,P_i)\xrightarrow{f^*}\Hom_A(P_1,P_i).
\]
Applying $\Hom_A(-,P_i)$ to \eqref{eq:principal-projres} gives
\[
 \Hom_{\mathcal H_Q}(Z_M,I_i)\simeq\Ext_A^1(M,P_i).
\]
Together with \eqref{eq:principal-m-hom}, this proves the first equality.

For the second, let $S$ be the Serre functor on $D^b(\modu A)$ and $\tau_D=S[-1]$.  Serre duality gives
\[
\begin{aligned}
 D\Ext_A^1(M,P_i)
 &=D\Hom_{D^b(A)}(M,P_i[1])\\
 &\simeq\Hom_{D^b(A)}(P_i[1],SM)\\
 &\simeq\Hom_{D^b(A)}(P_i,\tau_D M).
\end{aligned}
\]
Since $A$ is hereditary and $M$ is non-projective, $\tau_D M\simeq\tau_A M$ is a module.  Hence $D\Ext_A^1(M,P_i)\simeq\Hom_A(P_i,\tau_A M)\simeq(\tau_A M)e_i$, and dimensions give the second equality.
\end{proof}

\begin{corollary}\label{cor:principal-AR}
Let $M$ be an indecomposable non-projective right $A$-module, and let
\[
 \tau Z_M\longrightarrow Y_M\longrightarrow Z_M\dashrightarrow
\]
be the Auslander--Reiten $\E$-triangle in $\mathcal H_Q$.  Then
\begin{equation}\label{eq:principal-AR}
 \X^{\boldsymbol\Gamma}_{\tau Z_M}\X^{\boldsymbol\Gamma}_{Z_M}
 =\X^{\boldsymbol\Gamma}_{Y_M}
  +\prod_{i\in Q_0}x_{i^+}^{\dim_{\kk}(\tau_A M)_i}.
\end{equation}
\end{corollary}

\begin{proof}
 The proof follows from \cref{cor:frozen-mesh} and \cref{prop:principal-multiplicity}.
\end{proof}

\begin{remark}[Comparison with Dupont character with coefficients]\label{rem:dupont-comparison}
Put $y_i=x_{i^+}$ and identify right $A=\kk Q$-modules with left $\kk Q^{\mathrm{op}}$-modules. Dupont \cite[Definition~2.1]{DupontCoefficients} defined the classical cluster character with coefficients, and Dupont \cite[Proposition~2.2]{DupontCoefficients} proved its multiplication formula for almost split sequences. With the usual suspension convention, that result gives the  identity corresponding to \eqref{eq:principal-AR}. In particular, it already covers the homogeneous tube case. We do not repeat Dupont's comparison or proof here. The additional information in \cref{prop:principal-multiplicity,cor:principal-AR} is the explicit calculation of the projective--injective object whose character realises the coefficient term in the Higgs category.
\end{remark}

\subsubsection{The Higgs category is genuinely extriangulated}
\label{subsec}

In the principal-coefficient case, the Higgs category provides a
categorification which is genuinely extriangulated.  In particular, it
does not reduce to either an exact Frobenius category or a triangulated
category.

\begin{proposition}
\label{prop}
Let $Q$ be an acyclic quiver and let $\mathcal H_Q$ be the corresponding
principal-coefficient Higgs category. Then $\mathcal H_Q$ is neither exact
nor triangulated.
\end{proposition}

\begin{proof}
We first prove that $\mathcal H_Q$ is not exact.
By Wu's construction, $\mathcal H_Q$ is a Frobenius extriangulated
category whose projective-injective objects are generated by the frozen
summands. Moreover, its stable category is triangle equivalent to the
corresponding Amiot generalized cluster category:
$$
\underline{\mathcal H_Q}\simeq \mathcal C(Q,W).
$$
The frozen summands encode the principal coefficients.

Suppose that $\mathcal H_Q$ were exact. Then the above Frobenius
extriangulated structure would give a Frobenius exact categorification of
the principal-coefficient cluster algebra with the usual cluster-tilting
object. By \cite{PresslandPrincipal}, Pressland proved that acyclic cluster algebras with
principal coefficients do not admit such a Frobenius exact
categorification satisfying the standard cluster-tilting conditions. Hence $\mathcal H_Q$ cannot be exact.

We can also see this obstruction intrinsically from the relative
Ginzburg model. In the present case, Wu's stalk criterion applies. The
frozen radical representability condition is automatic, whereas the
relative Ginzburg algebra has a non-trivial negative-degree corner:
$$
H^{-1}(e_{i^+}\Gamma_{\mathrm{rel}}e_i)\neq 0 .
$$
Therefore the relative Ginzburg algebra is not a stalk algebra, and Wu's
criterion implies again that $\mathcal H_Q$ is not exact.

It remains to show that $\mathcal H_Q$ is not triangulated.  The category
$\mathcal H_Q$ contains non-zero projective-injective objects, namely the
frozen summands.  If a Frobenius extriangulated category were induced by
a triangulated category, then for a projective object $P$ we would have
$$
\mathbb E(P,X)=0
$$
for all objects $X$.  Since in a triangulated category
$$
\mathbb E(P,X)\simeq \operatorname{Hom}(P,\Sigma X),
$$
taking $X=\Sigma^{-1}P$ gives
$$
\operatorname{Hom}(P,P)=0.
$$
Hence $P=0$.  Therefore a triangulated Frobenius category has no
non-zero projective objects.  Since $\mathcal H_Q$ has non-zero
projective-injective frozen summands, it cannot be triangulated.
\end{proof}

\subsubsection{Example: Affine homogeneous tubes}
Assume that $Q$ is affine.  Let $R_\lambda$ be a quasi-simple in a homogeneous tube of $\modu A$, $\delta=(\delta_i)_{i\in Q_0}$  the minimal positive imaginary root and put $Z_\lambda=Z_{R_\lambda}$.

A homogeneous tube has rank one, so $\tau_A R_\lambda\simeq R_\lambda$. Over $\kk$, its quasi-simple has dimension vector $\delta$ and endomorphism ring $\kk$ by Ringel \cite{RingelTame}.  Hence \cref{prop:principal-multiplicity} gives $m_i(Z_\lambda)=\delta_i$ and therefore $I_{Z_\lambda}\simeq\bigoplus_iI_i^{\delta_i}$.

By \cref{lem:principal-lift-minimal}, $Z_\lambda$ is indecomposable and non-projective and its stable image is $R_\lambda$.  In the orbit category $\tau_D^{-1}\Sigma$ becomes isomorphic to the identity.  Hence
\[
 \Sigma R_\lambda\simeq \tau_D R_\lambda
 \simeq \tau_A R_\lambda\simeq R_\lambda.
\]
Since $\pi(\tau Z_\lambda)\simeq\Sigma\pi Z_\lambda$, \cref{lem:stable-lift} gives $\tau Z_\lambda\simeq Z_\lambda$ in $\mathcal H_Q$.

We next compute the extension dimension.  By the formula for the orbit category of Buan, Marsh, Reineke, Reiten and Todorov \cite{BMRRT},
\[
 \End_{\mathcal C_Q}(R_\lambda)
 =\bigoplus_{m\in\mathbb Z}
 \Hom_{D^b(A)}\!\left(R_\lambda,
 (\tau_D^{-1}\Sigma)^mR_\lambda\right).
\]
All objects $\tau_D^{-m}R_\lambda$ are regular $A$-modules and $A$ is hereditary.  Therefore the summands with $m<0$ vanish for degree reasons, those with $m\ge2$ vanish because $\Ext_A^m=0$, and only $m=0,1$ remain.  Thus
\[
 \End_{\mathcal C_Q}(R_\lambda)
 \simeq \End_A(R_\lambda)
 \oplus\Ext_A^1(R_\lambda,\tau_A^{-1}R_\lambda).
\]
Since $\tau_A R_\lambda\simeq R_\lambda$, the second summand is $\Ext_A^1(R_\lambda,R_\lambda)$.  The hereditary Euler form gives
\[
 \dim_{\kk}\Hom_A(R_\lambda,R_\lambda)
 -\dim_{\kk}\Ext_A^1(R_\lambda,R_\lambda)
 =\langle\delta,\delta\rangle=0.
\]
As $\End_A(R_\lambda)\simeq\kk$, it follows that
\[
 \dim_{\kk}\Ext_A^1(R_\lambda,R_\lambda)=1,
 \hspace{2.5mm}
 \dim_{\kk}\End_{\mathcal C_Q}(R_\lambda)=2.
\]
Thus the space of self extensions in $\modu A$ is one-dimensional.

Finally, we use $\pi(\tau Z_\lambda)\simeq\Sigma R_\lambda$.  The $2$-Calabi--Yau property of $\mathcal C_Q$ gives
\[
\begin{aligned}
 \E_{\mathcal H_Q}(Z_\lambda,\tau Z_\lambda)
 &\simeq\Hom_{\mathcal C_Q}
      (R_\lambda,\Sigma\pi(\tau Z_\lambda))\\
 &\simeq\Hom_{\mathcal C_Q}(R_\lambda,\Sigma^2R_\lambda)\\
 &\simeq D\End_{\mathcal C_Q}(R_\lambda).
\end{aligned}
\]
Hence $\dim_{\kk}\E_{\mathcal H_Q}(Z_\lambda,\tau Z_\lambda)=2$.

If $\tau Z_\lambda\to Y_\lambda\to Z_\lambda\dashrightarrow$ is the Auslander--Reiten $\E$-triangle, by \cref{cor:principal-AR} and $\tau Z_\lambda\simeq Z_\lambda$, then
\begin{equation}\label{eq:affine-principal}
 \bigl(\X^{\boldsymbol\Gamma}_{Z_\lambda}\bigr)^2
 =\X^{\boldsymbol\Gamma}_{Y_\lambda}+\prod_{i\in Q_0}x_{i^+}^{\delta_i}.
\end{equation}

By \cref{rem:dupont-comparison}, \eqref{eq:affine-principal} agrees with the homogeneous tube case of the formula proved by Dupont \cite[Proposition~2.2]{DupontCoefficients}.

\medskip
\hspace{-5.5mm}\textbf{Acknowledgements:} Ming Ding and Fan Xu are supported by the National Natural Science Foundation of China (Grant No. 12371036). Panyue Zhou is supported by the National Natural Science Foundation of China (Grant No.~12371034).
\newpage

\hspace{-5.5mm}{\bf Statement on AI:}\hspace{2mm} ChatGPT was used only for language editing and improving the clarity of the manuscript. All mathematical results, arguments, and proofs were developed and verified by the authors.
\vspace{3mm}

\hspace{-5.5mm}\textbf{Data Availability:}\hspace{2mm} Data sharing not applicable to this article as no datasets were generated or analysed during
the current study.
\vspace{3mm}

\hspace{-5.5mm}\textbf{Conflict of Interests:}\hspace{2mm} The authors declare that they have no conflicts of interest to this work.

\hspace{-6.5mm}
\textbf{Ming Ding}\\[1mm]
School of Mathematics and Information Science,
Guangzhou University, Guangzhou 510006, P. R. China\\[1mm]
Email: dingming@gzhu.edu.cn
\vspace{3mm}

\hspace{-6.5mm}
\textbf{Fan Xu}\\[1mm]
Department of Mathematical Sciences,
Tsinghua University,
Beijing 100084, P.~R.~China\\[1mm]
E-mail: fanxu@mail.tsinghua.edu.cn
\vspace{3mm}

\hspace{-6.5mm}
\textbf{Panyue Zhou}\\[1mm]
School of Mathematics and Statistics, Changsha University of Science and Technology, Changsha 410114, Hunan, P. R. China\\[1mm]
Email: panyuezhou@163.com\\

\end{document}